\documentclass[11pt,leqno]{article}

\usepackage{amsthm,amsfonts,amssymb,amsmath,oldgerm,mathtools}
\usepackage{enumerate}
\usepackage{fullpage}
\usepackage{graphicx}
\usepackage{verbatim}
\usepackage{mathrsfs}
\usepackage[dvipsnames]{xcolor}
\usepackage{xcolor}
\usepackage[colorinlistoftodos]{todonotes} 
\usepackage[italicdiff]{physics}
\usepackage[hidelinks]{hyperref}
\usepackage[noabbrev]{cleveref}
\usepackage[titletoc, title]{appendix}
\usepackage[numbers,sort&compress]{natbib} 

\numberwithin{equation}{section}

\renewcommand\d{\partial}

\newcommand{\R}{\mathbb R}
\newcommand{\C}{\mathbb C}

\newcommand\nn{\nonumber}

\newcommand\LA{\left\langle}
\newcommand\RA{\right\rangle}

\newcommand{\sign}{{\rm sgn }}

\newcommand{\Span}{{\rm span}}
\newcommand{\RM}{{\mathbb{R}}}
\newcommand{\CM}{{\mathbb{C}}}

\newtheorem{theorem}{Theorem}[section]
\newtheorem{proposition}[theorem]{Proposition}

\newtheorem{lemma}[theorem]{Lemma}
\newtheorem{remark}[theorem]{Remark}
\theoremstyle{definition}

\newtheorem{assumption}[theorem]{Assumption}

\crefname{assumption}{assumption}{assumptions}
\Crefname{appsec}{appendix}{appendices}
\allowdisplaybreaks[3]

\title{Rigorous justification of the Whitham modulation equations for equations of Whitham-type}

\author{William A. Clarke\thanks{School of Mathematics and Statistics, University of Sydney, NSW 2006, Australia; \texttt{wcla7359@uni.sydney.edu.au}}
\quad Robert Marangell\thanks{School of Mathematics and Statistics, University of Sydney, NSW 2006, Australia; \texttt{robert.marangell@sydney.edu.au}}
\quad Wesley R. Perkins\thanks{Department of Mathematics, Lehigh University, 17 Memorial Drive East, Bethlehem, PA 18015, USA; \texttt{wrp221@lehigh.edu}}
}

\date{\today}

\begin{document}

\maketitle






\begin{abstract} \noindent We prove that the modulational instability criterion of the formal Whitham modulation theory agrees with the spectral stability of long wavelength perturbations of periodic travelling wave solutions to the generalized Whitham equation. We use the standard WKB procedure to derive a quasi-linear system of three Whitham modulation equations, written in terms of the mass, momentum, and wave number of a periodic travelling wave solution. We use the same quantities as parameters in a rigorous spectral perturbation of the linearized operator, which allows us to track the bifurcation of the zero eigenvalue as the Floquet parameter varies. We show that the hyperbolicity of the Whitham system is a necessary condition for the existence of purely imaginary eigenvalues in the linearized system, and hence also a prerequisite for modulational stability of the underlying wave. Since the generalized Whitham equation has a Hamiltonian structure, we conclude that strict hyperbolicity is a sufficient condition for modulational stability.
\end{abstract}

\section{Introduction}\label{S:intro}

We consider the modulational stability of periodic travelling wave solutions to the generalized Whitham equation
\begin{align}\label{e:Whitham}
    u_t + f(u)_x + \int_{\R}K(\xi)u_x(x-\xi,t)d\xi = 0.
\end{align}
This equation was originally posed by Whitham \cite{Whitham1967,Whitham1999} as a way of modelling both the full linear dispersion of water waves and the breaking and peaking effects from the shallow water theory. Defining the convolution as the operator $ \mathcal{K}* $, the solutions $ u(x,t) = e^{i(kx -\omega t)} $ will solve
\begin{align*}
	u_{t} + \mathcal{K}*u_{x} = 0
\end{align*}
provided that 
\begin{align*}
	c &= \frac{\omega}{k} = \int_{\R}K(\xi)e^{-ik\xi}d\xi.
\end{align*}
This means that we can choose an arbitrary phase velocity $ c(k) $ by requiring $ K(\xi) $ to be its Fourier transform:
\begin{align*}
	K(x) &= \frac{1}{2\pi}\int_{\R}c(k)e^{ikx}dk.
\end{align*}
The kernel from the water wave problem corresponds to the choice
\begin{align}
	c_{W}(k) = \sqrt{\frac{\tanh k}{k}}, \label{eq:WhithamWater}
\end{align}
and a second order approximation yields
\begin{align*}
	c_{KdV}(k) = 1 - \frac{1}{6}k^{2} + {\rm O}(k^{4}).
\end{align*}
With this choice, the kernel becomes
\begin{align*}
	K_{KdV}(x) = \delta(x) + \frac{1}{6}\delta''(x),
\end{align*}
which corresponds to the KdV equation:
\begin{align}
	u_{t} + f(u)_{x} + u_{x} + \frac{1}{6}u_{xxx} = 0.\label{eq:KdV}
\end{align}
The symbol $ c_{KdV} $ is not a good approximation of $ c_{W} $ for large wavenumbers, perhaps explaining why \Cref{eq:KdV} fails to capture some behaviour of water waves such as breaking and peaking, even if it does support solitary waves. On the other hand, proofs of wave breaking \cite{Hur2017}, peaking \cite{EW2019} and solitary wave solutions \cite{EGW12}\footnote{In fact, the authors prove the existence of solitary waves for a large class of Whitham-type equations with $ c(k) $ even and nonlinearity $ f(x) \sim |x|^{p} $.} have all been established for \Cref{e:Whitham} with $ c = c_{W} $. A number of other local and nonlocal equations take the form of \Cref{e:Whitham}, including (but not limited to): Burgers' equation $ c(k) = ik $; the Kawahara equation $ c(k) = -k^{2} + k^{4} $; the intermediate long wave equation $ c(k) = k \coth(\delta k) - \frac{1}{\delta} $ (for a constant $ \delta > 0 $); and the Benjamin-Ono equation $ c(k) = k\;\sign\; k $.

This paper focuses on the modulational stability of periodic travelling wave solutions of \Cref{e:Whitham}, namely their spectral stability to long wavelength perturbations. Whitham modulation theory provides a formal procedure for deriving a modulational instability criterion by considering the linear stability of a modulational approximation of the original PDE written in terms of conserved quantities (or wave parameters) \cite{Whitham1999,Whitham1970,W65}.  Whitham modulation theory has been widely used to study many physical phenomena, such as dispersive shock waves (DSWs); for a survey of recent results, see \cite{EH2016}.  Additionally, one topic of current interest is the ability to use DSW fitting and the system of equations provided by Whitham's theory to find DSW solutions, including when the system cannot be set in Riemann invariant form, such as the Boussinesq Benjamin-Ono and Benjamin-Bona-Mahony equations, see \cite{CEHS2021,NS2021}, respectively.
At present, there is no proof of the agreement of Whitham's formal theory with a rigorous spectral analysis at the origin for a general PDE. There is, however, a growing body of evidence that general dispersive PDEs allow these two theories to coincide. In particular, proofs exist for: systems of viscous conservation laws \cite{OZ2006,Serre05}; the generalized KdV equation \cite{JZ2010}; the nonlinear Klein-Gordon equation \cite{JMMP2014}; systems of non-dissipative, local Hamiltonian equations \cite{BNR14}; a viscous fluid conduit equation \cite{JP2020}; and the generalized nonlinear Schr{\"o}dinger equation \cite{CM2021}. The primary goal of this paper is to present the first proof of the agreement between Whitham modulation theory and spectral stability to long wavelength perturbations for the generalized (nonlocal) Whitham \Cref{e:Whitham}, a primary benefit of which is the concise justification of Whitham modulation theory for all equations that can be expressed in the form \Cref{e:Whitham} satisfying the following assumptions. 
\subsection*{Assumptions}
\begin{assumption}
	$ K(\xi) \in L^{1}(\R)$ is an even, real-valued function which satisfies $ \xi K(\xi), \xi^{2}K(\xi),  \xi^{3}K(\xi) \in L^{1}(\R)$. Further, $ \widehat{K}(q) = c(q) = \frac{\Omega(q)}{q} $, where $ \Omega(q) $ is an odd, real-valued function. \label{assumption:a1}
\end{assumption}
\begin{assumption}
	\label{assumption:a2}  $f(u)$ is smooth.
\end{assumption}
\begin{assumption}
	\label{assumption:a3}  There exists a three parameter family of translation-invariant $2\pi$-periodic travelling wave solutions $\phi(\theta + x_0;{\bf p})$, where $\theta=k(x -ct)$ is the traveling coordinate with wavenumber $k$ and wavespeed $c=c({\bf p})$. We define ${\bf p} := (k, M, P)$ where $ M,P $ are the mass and momentum of $ \phi $ over a period (defined in \Cref{e:M,e:P} respectively), and we suppose that $ {\bf p} \in U\subseteq \R^3$ for some open set $ U $. 
\end{assumption}
\begin{assumption}\label{assumption:a4}We fix $x_0$ so that $\phi$ is an even function. We further assume that $ \phi $ is non-constant and $\ker_{\rm per}(\mathcal{L}[\phi]) = \Span\{\phi'\}$, where $ \mathcal{L}[\phi] $ is the linear operator defined in \Cref{eq:linearoperators}.
\end{assumption}

\Cref{assumption:a1,assumption:a2,assumption:a3} are inspired by \cite{BIHS2021}, with the modification that \Cref{assumption:a1} is weaker than the corresponding assumption in \cite{BIHS2021}, where the analyticity of $ \Omega(q) $ in fact implies that $ K(\xi) $ is a function of rapid decay. We also only use \Cref{assumption:a2} implicitly when we cite \cite{BIHS2021}.

\Cref{assumption:a1,assumption:a2,assumption:a3} are satisfied by many equations including the Whitham equation, generalized KdV, the Fornberg-Whitham equation, and the intermediate long wave equation. Note that in the case of KdV and Fornberg-Whitham, $ K(\xi) $ is a function of $ \delta(\xi) $ and its derivatives, which are (technically) not in $ L^{1}(\R) $. However, since their integral over $ \R $ exists in a distributional sense, we can (and do) relax the Lebesgue integrability assumption to an assumption that the appropriate integrals exist. Periodic travelling wave solutions exist for the Whitham equation with $ c_{W} $ defined in \Cref{eq:WhithamWater} \cite{EK09,EK13}. Proving existence for \Cref{e:Whitham} lies beyond the scope of this paper, although the authors of \cite{BH14} argue that the existence of periodic travelling waves often follows from variational arguments for the particular equations concerned. Lastly, we make \Cref{assumption:a4} in line with \cite[Assumption 5.N2]{BH14} as it allows for a non-degenerate characterization of the kernel of the linearized operator  and is satisfied by the linearized operators for many equations. We omit an assumption similar to \cite[Assumption 5.N3]{BH14}, namely that $ M_{c}P_{b} - M_{b}P_{c} \neq 0 $ ($ b $ is defined in \Cref{e:profile2}), since a similar condition appears in the proof of \Cref{P:gkerper} as a {\it consequence} of \Cref{assumption:a4}.

\begin{remark}
    To put \Cref{assumption:a3} into a larger context, note that, in many applications, existence can be proven in terms of the mathematical parameters $a$, $E$, and $c$, where $a$ and $E$ are constants of integration arising from reducing the profile equation to quadrature, see for example \cite{JP2020,JZ2010}.  In such papers, the results require the assumption
    \[
    \pdv{(k,M,P)}{(a,E,c)} = -\frac{1}{T^2}\pdv{(T,M,P)}{(a,E,c)}\neq0,
    \]
    where $T=1/k$ is the period of the background wave.  (Note that this assumption is a natural extension of the assumption that $\pdv{P}{c}>0$ when studying the stability of solitary waves.)  The Implicit Function Theorem then guarantees that we can reparametrize any waves shown to exist in the $(a,E,c)$ coordinate system in terms of the natural Whitham parameters $(k,M,P)$.  In general, there needs to be some way of comparing the results of the existence theory, which is usually executed in terms of the mathematical parameters availalble, with the results of Whitham's theory, which is usually parametrized by the wave number, mass, and momentum.  In this paper, we assume that this reparametrization has already occurred in \Cref{assumption:a3}.  
\end{remark}

\begin{remark}
    Throughout the remainder of our paper, we will work almost exclusively from the perspective that $\mathcal{K}*$ is a convolution operator, rather than the equivalent notion of $\mathcal{K}*$ as a Fourier multiplier.  A motivation for this decision is that we are able to weaken the assumptions of \cite{BIHS2021} and thereby extend their result to include a wider class of dispersive equations.  This is done in \Cref{appendix:aA} through the careful tracking of remainder estimates in the use of Taylor's theorem.  As another motivation, the derivation of the Whitham modulation equations using Fourier multipliers unfortunately breaks down when the equations are elliptic, but a similar break down does not occur when considering convolution operators.  While this isn't relevant to modulational stability (since we want the equations to be hyperbolic), we still note that it broadens the derived Whitham modulation equations to possibly be elliptic.  
\end{remark}
\medskip

The outline of the paper is as follows.  In \Cref{S:prelim}, we give some preliminary  results that allow us to more completely understand the manifold of $2\pi$-periodic solutions given in \Cref{assumption:a3}.  In \Cref{S:WModEqns}, we express the system of Whitham modulation equations in quasilinear form and state our main result, \Cref{T:main}.  In Section \ref{S:RigorousTheory}, we prove \Cref{T:main} through the use of rigorous spectral perturbation theory.  More specifically, we analyze the spectrum of the linearized operator associated to \Cref{e:Whitham} near the origin of the spectral plane, which we then connect to the quasilinear Whitham modulation equations.  Finally in \Cref{appendix:aA} , we provide a rederivation of the Whitham modulation equations using a weaker set of assumptions than those given in \cite{BIHS2021}.
\medskip

\noindent
{\bf Acknowledgments:} The authors would like to thank Prof. Mat Johnson for his help to facilitate our collaboration and his insightful comments on spectral perturbation theory. WC acknowledges the support of the Australian Research Council under grant DP200102130. RM acknowledges the support of the Australian Research Council under grant DP210101102.

\section{Preliminaries}\label{S:prelim}
In this section, we outline the properties of $ \mathcal{K}* $ which allow us to parametrize the manifold of $ 2\pi $-periodic travelling wave solutions with the quantities $ k,M,P $.
\begin{lemma}\label{L:Kprop}
When acting on $ L^2_{\rm per}[0,2\pi) $, the operator $\mathcal{K}*$ is symmetric and commutes with the differential operators $\d_x,\d_{t}$.
\end{lemma}
\begin{proof}
	To prove that $ \mathcal{K}* $ is symmetric, first let $ u(x,t),v(x,t) \in L^{2}_{\rm per}[0,2\pi)\times C^{1}$. Further, for $ \phi(x,t) $, a $ 2\pi $-periodic function in $ x $,  we write its Fourier series expansion as
	\begin{align*}
		\phi(x,t) &= \sum_{n=-\infty}^{\infty}\phi_{n}(t)e^{inx}\\
		\phi_{n} &= \frac{1}{2\pi}\int_{-\pi}^{\pi}\phi(x,t)e^{-inx}dx.
	\end{align*}
	The convolution $ \mathcal{K}*\phi $ is $ 2\pi $-periodic in $ x $, so we calculate its Fourier coefficients as
	\begin{align*}
		c_{n}(t) &= \frac{1}{2\pi}\int_{-\pi}^{\pi}\int_{\R}K(\xi)\phi(x-\xi,t)e^{-inx}d\xi dx\\
		&= \int_{\R}K(\xi)\phi_{n}(t)e^{-in\xi}d\xi\\
		&= \frac{\Omega(n)}{n}\phi_{n}(t).
	\end{align*}
	So by Parseval's theorem (with the standard $ L^{2}[0,2\pi) $ inner product) we have
	\begin{align*}
		\LA u,\mathcal{K}*v \RA &= 2\pi\sum_{n=-\infty}^{\infty}u_{n}(t)\overline{\frac{\Omega(n)}{n}v_{n}(t)}\\
		&= 2\pi\sum_{n=-\infty}^{\infty}\frac{\Omega(n)}{n}u_{n}(t)\overline{v_{n}(t)}\\
		&= \LA \mathcal{K}*u,v \RA.
	\end{align*}
	To justify $ \partial_{x}\mathcal{K}*\phi = \mathcal{K}*\phi_{x} $, we note that $ g(t) = \max_{x\in[0,2\pi)}\phi_{x}(x,t) $ dominates $ \phi_{x}(x,t) $. If $ K(\xi) \in L^{1}(\R) $, then the dominated convergence theorem gives us the desired result. There are separate arguments for the cases where the convolution is not a Lebesgue integral, for instance the Benjamin-Ono equation. Under appropriate hypotheses, the integral will converge uniformly and the derivative operator $ \partial_{x} $ may be interchanged with the integral (cf. \cite{King09} for a discussion pertaining to the Hilbert transform). The same argument also applies to the derivative operator $ \partial_{t} $, by noting that we are interested in stationary solutions to  \Cref{eq:stationary} so we may take $ u(x,t) $ to be bounded in $ t $.
\end{proof}

For $u\in L^2_{\rm per}[0,2\pi)$, \Cref{e:Whitham} admits the following conserved quantities \cite{BH14,NS1994}:
\[
M = \int_0^{2\pi} u\, dx, \qquad P = \int_0^{2\pi} \frac{1}{2}u^2\, dx,\qquad H = \int_0^{2\pi} \left[F(u) + \frac{1}{2}u\mathcal{K}*u\right]\,dx,
\]
where $F' = f$ . The existence of a Hamiltonian allows us to write \Cref{e:Whitham} in the variational form
\[
u_t = \mathcal{J} \fdv{H}{u}()(u), \qquad \mathcal{J}=\d_x,
\]
whence the results of \cite{BH14} apply. However, by \Cref{assumption:a3} we consider a $ 2\pi $-periodic family of travelling wave solutions parametrized by $ (k,M,P) $, whereas the authors of \cite{BH14} use the parameters $ (c,a,T) $ where $ a $ is a constant of integration and $ T $ is the period of a travelling wave solution. We make our choice because the derivation of the Whitham modulation equations involves fixing the period, and $ (k,M,P) $ are a particularly convenient set of variables given that the Whitham modulation equations describe the long-time evolution of the parameters $(k,M,P)$.

Returning to \Cref{e:Whitham}, \Cref{assumption:a3} implies that $\phi(\theta;k,M,P)$ is a $2\pi$-periodic stationary solution to the equation
\begin{equation}\label{eq:stationary}
	u_t - kc u_\theta + k[f(u)]_\theta + \mathcal{K}*(ku_\theta) = 0.
\end{equation} 
Note that the convolution has the form:
\begin{align}
	\mathcal{K}*u_\theta &= \int_{\R}K(\xi)u_{\theta}(\theta-k\xi)d\xi. \label{eq:thetaconv}
\end{align}
We conclude that $\phi$ satisfies the profile equation
\begin{equation}\label{e:profile}
	k\d_\theta\left[-c\phi + f(\phi) + \mathcal{K}*\phi\right]=0,
\end{equation}
and integrating once yields
\begin{equation}\label{e:profile2}
    -kc\phi + kf(\phi) + k\mathcal{K}*\phi = b
\end{equation}
for some constant of integration $b$ which is parametrized by $ (k,M,P) $.  
The conserved quantities may now be expressed as
\begin{equation}\label{e:M}
    M = \int_0^{2\pi} \phi(\theta;k,M,P)\,d\theta
\end{equation}
and
\begin{equation}\label{e:P}
    P = \int_0^{2\pi} \frac{1}{2}\phi^2(\theta;k,M,P)\,d\theta.
\end{equation}

\section{The Whitham Modulation Equations}\label{S:WModEqns}
The procedure for deriving the Whitham modulation equations involves substituting the ansatz
\begin{align*}
	u(x,t) &= u_{0}(\theta,X,T) + \epsilon u_{1}(\theta,X,T) + \dots,\qquad \theta = \frac{1}{\epsilon}\psi(X,T), X=\epsilon x, T=\epsilon t
\end{align*}
into \Cref{e:Whitham} and collecting terms at each power of $ \epsilon $. Specifically, the first order modulation equations are the solvability conditions for the linearized operator of the travelling wave ODE as determined by the Fredholm alternative; this approach is equivalent to the suppression of secular terms in the expansion, or the application of the variational principle to an averaged Lagrangian \cite{L66,Whitham1970,Whitham1999}. At O$ (1) $, $ u_{0} $ satisfies the travelling wave ODE \Cref{e:profile}, so by \Cref{assumption:a3} we restrict its form to:
\begin{align*}
	u_{0}(\theta,X,T) &= u_{0}(\theta,k(X,T),M(X,T),P(X,T)).
\end{align*}
In \cite{BIHS2021}, the authors derive the first order Whitham modulation equation under the assumption that $ \Omega $ is analytic. We demonstrate an alternative but equivalent derivation in \Cref{appendix:aA} which does not rely on the analyticity of $ \Omega $. The Whitham modulation equations are governed by the system \cite{BIHS2021}
\begin{equation}\label{e:whitmodsystem}
\begin{cases}
k_T = -(kc)_X\\
M_T = -\LA 1, f(u_0)_{X} + \mathcal{K}*u_{0X}\RA\\
P_T = -\LA u_{0}, f(u_{0})_{X}\RA - \frac{1}{2}\LA u_0, \mathcal{K}^{(1)}*u_{0X} + (\mathcal{K}^{(1)}*u_{0})_{X}\RA
\end{cases},
\end{equation}
where $\mathcal{K}^{(1)}$ is a Fourier multiplier defined by the symbol
\begin{equation}\label{e:K1}
(\widehat{\mathcal{K}^{(1)}*g})(q) := \Omega'(kq)\widehat{g}(q).
\end{equation}
Alternatively, we can treat $\mathcal{K}^{(1)}$ as the convolution
\[
	\mathcal{K}^{(1)}*g = \int_{\R}-\xi K'(\xi)g(\theta - k\xi,X,T)d\xi,
\]
where we understand $ K'(\xi) = -K(\xi)\partial_{\xi} $ in a distributional sense. Using integration by parts, we see that
\begin{align}
	\mathcal{K}^{(1)}*g &= \int_{\R}K(\xi)g(\theta - k\xi,X,T) - k\xi K(\xi)g_{\theta}(\theta - k\xi,X,T)d\xi\nn\\
	&= \mathcal{K}*g - k\mathcal{K}_{1}*g_{\theta},\nn\\
	\mathcal{K}_{1}*\phi &:= \int_{\R}\xi K(\xi)\phi(\theta-k\xi,X,T)d\xi. \label{eq:K_1}
\end{align}
Similarly to the proof of \Cref{L:Kprop}, we note that $\mathcal{K}^{(1)}*$ is a symmetric operator. Our aim is to use the chain rule
\begin{align}
	\partial_{X} = k_{X}\partial_{k} + M_{X}\partial_{M} + P_{X}\partial_{P} \label{eq:chainrule}
\end{align}
in \Cref{e:whitmodsystem} in order to express this system in the quasilinear form
\begin{equation}\label{e:whitmod_quasilinear}
\left(\begin{array}{c}
    k \\
    M \\
    P
\end{array}\right)_T = {\bf D}(u_0)
\left(\begin{array}{c}
    k \\
    M \\
    P
\end{array}\right)_X.
\end{equation}
The only term which requires particular attention is:
\begin{align*}
	(\mathcal{K}^{(1)}*u_{0})_{X} &= \partial_{X}\left(\mathcal{K}*u_{0} - k\mathcal{K}_{1}*u_{0\theta}\right)\\
	&= \partial_{X}\int_{\R}K(\xi)u_{0}(\theta - k\xi,X,T) - k\xi K(\xi)u_{0\theta}(\theta - k\xi,X,T)d\xi\\
	&= \int_{\R}K(\xi)\left(-k_{X}\xi u_{0\theta} + u_{0X}\right) - k_{X}\xi K(\xi)u_{0\theta} - k\xi K(\xi)\left(-k_{X}\xi u_{0\theta\theta} + u_{0\theta X}\right)d\xi\\
	&= \mathcal{K}*u_{0X} - 2k_{X}\mathcal{K}_{1}*u_{0\theta} - k\mathcal{K}_{1}*u_{0\theta X}+ kk_{X}\mathcal{K}_{2}*u_{0\theta\theta}
\end{align*}
where we define
\begin{align}
	\mathcal{K}_{2}*g &:= \int_{\R}\xi^{2}K(\xi)g(\theta - k\xi,X,T)d\xi.\label{eq:K_2}
\end{align}
Hence we may write
\begin{align*}
	\mathcal{K}^{(1)}*u_{0X} + (\mathcal{K}^{(1)}*u_{0})_{X} &= 2\mathcal{K}*u_{0X} - 2k\mathcal{K}_{1}*u_{0\theta X} - 2k_{X}\mathcal{K}_{1}*u_{0\theta} + kk_{X}\mathcal{K}_{2}*u_{0\theta\theta}.
\end{align*}
We intend to linearize this system around the travelling wave profile $ \phi(\theta;k_{0},M_{0},P_{0}) $, and so we can justify the commutativity of $ \partial_{X} $ and the integral by restricting $ u_{0} $ to be bounded in $ X $ with continuous second order derivatives. This is discussed in more detail in \Cref{lemma:app1}. Importantly, the above calculation agrees with the results in \Cref{appendix:aA} which do not rely on boundedness, and so we proceed with this expansion. We now use the chain rule in \Cref{eq:chainrule}, so that
\begin{align*}
	u_{0X} &= k_{X}u_{0k} + M_{X}u_{0M} + P_{X}u_{0P},
\end{align*}
and similarly for higher order derivatives. Upon substitution into \Cref{e:whitmod_quasilinear}, we arrive at
\begin{align}
	{\bf D}(u_{0}) = \begin{psmallmatrix}
		-kc_k - c & -kc_M & -kc_P \\
		-\LA 1,f'(u_{0})u_{0k} + \mathcal{K}*u_{0k}\RA & -\LA 1,f'(u_{0})u_{0M} + \mathcal{K}*u_{0M}\RA & -\LA 1,f'(u_{0})u_{0P} + \mathcal{K}*u_{0P}\RA \\
		 d_{31} & d_{32} & d_{33}
	\end{psmallmatrix},\label{eq:WhithamMatrix}
\end{align}
where
\begin{align*}
	d_{31} &= -\LA u_{0}, f'(u_{0})u_{0k} + \mathcal{K}*u_{0k} - k\mathcal{K}_{1}*u_{0\theta k} - \mathcal{K}_{1}*u_{0\theta} + \frac{1}{2}k\mathcal{K}_{2}*u_{0\theta\theta}\RA\\
	d_{32} &= -\LA u_{0}, f'(u_{0})u_{0M} + \mathcal{K}*u_{0M} - k\mathcal{K}_{1}*u_{0\theta M}\RA\\
	d_{33} &= -\LA u_{0}, f'(u_{0})u_{0P} + \mathcal{K}*u_{0P} - k\mathcal{K}_{1}*u_{0\theta P}\RA.
\end{align*}
\begin{theorem}\label{T:main}
Suppose $\phi_0$ satisfies \Cref{assumption:a1,assumption:a2,assumption:a3,assumption:a4}.  
A necessary condition for $\phi_0$ to be spectrally stable is that the Whitham modulation system is weakly hyperbolic at $(k_0, M(\phi_0), P(\phi_0))$, in the sense that all the characteristic speeds must be real.  Additionally, a sufficient condition for $\phi_0$ to be spectrally stable (in a neighborhood of the origin) is that the Whitham modulation system is strictly hyperbolic.
\end{theorem}
Before we prove this theorem, we recast the problem from the perspective of a rigorous spectral analysis of the linearized operator.
\section{Rigorous Modulation Stability Theory}\label{S:RigorousTheory}
In this section, we adapt the rigorous modulational stability theory of \cite{BNR14,JP2020} to \Cref{e:Whitham}. In particular, we find the dual right and left bases of the generalized kernel of the linearized operator. An extension of Floquet-Bloch theory to non-local eigenvalue problems with periodic coefficients \cite{J13} allows us to characterize the $ L^{2}(\R) $ spectrum of the linearized operator as the union of discrete spectra of a family of \textit{Bloch operators} $ \mathcal{A}_{\tau}[\phi] $ parametrized by the \textit{Bloch parameter} $ \tau $. Using spectral perturbation theory, we then construct a matrix the determinant of which captures the asymptotic behaviour of the spectrum of the linearized operator near the origin.
\subsection{Linearized eigenvalue problem}
The standard method of setting up this spectral problem, as seen in \cite{BNR14,BrHJ16,JP2020}, involves writing a nearby solution as $u(\theta,t) = \phi(\theta) + \epsilon v(\theta,t)$. Then with $ \epsilon \ll 1 $, at O$ (\epsilon) $ the perturbation $v(t)\in L^2(\R)$ satisfies
\[
v_t = \mathcal{A}[\phi]v,
\]
where
\begin{align}
	\mathcal{A}[\phi] = \d_\theta\mathcal{L}[\phi], \qquad \mathcal{L}[\phi] = k\left(c - f'(\phi) - \mathcal{K}*\right) \label{eq:linearoperators}
\end{align}
is considered as a densely defined operator acting on $L^2(\R)$. Note that $\mathcal{L}[\phi]$ is a symmetric operator by \Cref{L:Kprop}. Since $ \mathcal{A}[\phi] $ has $2\pi$-periodic coefficients and noting the extension of Floquet-Bloch theory to operators involving Fourier multipliers \cite[Proposition 3.1]{J13}, we conclude in the standard way that bounded solutions $ v $ to
\begin{equation}\label{eq:spectralproblem}
\mathcal{A}[\phi]v=\lambda v
\end{equation}
 are not in $ L^{2}(\R) $ (cf. \cite{KP_book,RS4}), but rather have the form
\begin{equation}\label{e:quasiper}
    v(\theta)=e^{i\tau \theta}w(\theta)
\end{equation}
for some $w\in L^2_{\rm per}[0,2\pi)$ and $\tau\in[-1/2,1/2)$.  This implies that $\lambda\in \CM$ belongs to the $L^2(\RM)$-spectrum of $A[\phi]$ if and only if there exists a $\tau\in[-1/2,1/2)$ and $w\in L^2_{\rm per}(0,2\pi)$ such that
\begin{equation}\label{bloch_spec}
\lambda w = e^{-i\tau \theta}\mathcal{A}[\phi]e^{i\tau \theta}w=:\mathcal{A}_\tau[\phi]w.
\end{equation}
A standard result of Floquet-Bloch theory is the decomposition of the essential spectrum of $ \mathcal{A}[\phi] $ into the union of isolated eigenvalues with finite algebraic multiplicities \cite{RS4}:
\begin{equation}\label{spec_decomp}
\sigma_{L^2(\RM)}\left(\mathcal{A}[\phi]\right)=\bigcup_{\tau\in[-1/2,1/2)}\sigma_{L^2_{\rm per}(0,2\pi)}\left(\mathcal{A}_\tau[\phi]\right),
\end{equation}

\subsection{Analysis of the Unmodulated Operator}
We are interested in the modulational stability of $ \phi $, namely the stability subject to long wavelength perturbations, which corresponds to $ \abs{\tau}\ll 1$, see \Cref{e:quasiper}. In order to describe this behaviour, we follow a procedure to determine the leading order behaviour of the spectral curves of $ \mathcal{A}_{\tau}[\phi] $ emerging from $ (\lambda,\tau) = (0,0) $. Firstly, we characterize the generalized kernel of the unmodulated operator $ \mathcal{A}_{0}[\phi] $, subject to \Cref{assumption:a3,assumption:a4}. Then we analyze how the generalized kernel breaks up into distinct eigenspaces as $ \tau $ varies; modulational instability arises when the eigenvalues corresponding to these eigenspaces have nonzero real parts. Our analysis and notation is almost identical to the results of \cite{BNR14,JP2020}.

We start by differentiating \Cref{e:profile2} with respect to $\theta$, $M$, and $P$, which yields
\begin{equation}\label{e:kernelelmts}
\mathcal{L}[\phi]\phi' = 0, \quad\mathcal{L}[\phi]\phi_M = -b_M - kc_M \phi, \quad \mathcal{L}[\phi]\phi_P = -b_P - kc_P \phi.
\end{equation}
We find the elements of the generalized kernel by differentiating again with respect to $ \theta $:
\begin{equation}\label{e:kernelelmts2}
\mathcal{A}[\phi]\phi'=0, \quad\mathcal{A}[\phi]\phi_M=-kc_M\phi', \quad\mathcal{A}[\phi]\phi_P=-kc_P\phi'.
\end{equation}
Since $ \mathcal{L}[\phi] $ is symmetric by \Cref{L:Kprop}, we have that
\[
\mathcal{A}^\dagger[\phi] = -\mathcal{L}[\phi]\d_\theta,
\]
which implies that
\begin{equation}\label{e:kernelelmts3}
\mathcal{A}^\dagger[\phi] \phi= 0 = \mathcal{A}^\dagger[\phi]1.
\end{equation}
We can find an element in the generalized kernel of $ \mathcal{A}^{\dagger}[\phi] $ by noting that
\[
\mathcal{A}^\dagger[\phi]\int_0^\theta \phi_M(z)\,dz = b_M +kc_M\phi,\quad\mathcal{A}^\dagger[\phi]\int_0^\theta \phi_P(z)\,dz = b_P +kc_P\phi.
\]
However, differentiating \Cref{e:M,e:P} with respect to $M$ and $P$ yields
\begin{equation}\label{e:innprod1}
    \LA 1, \phi_M\RA = 1 = \LA \phi, \phi_P\RA \quad\text{and}\quad \LA 1, \phi_P\RA = 0 = \LA \phi, \phi_M\RA,
\end{equation}
so only $\int_0^\theta \phi_P(z)dz$ is $2\pi$-periodic. 

\begin{proposition}\label{P:gkerper}
Under \Cref{assumption:a1,assumption:a2,assumption:a3,assumption:a4}, we have that $\lambda = 0$ is an eigenvalue of $\mathcal{A}_0[\phi]$ with algebraic multiplicity three and geometric multiplicity two.  Using the notation
\[
\begin{array}{ccccc}
    \Phi_1^0 = \phi' & ~~ & \Phi_2^0 = \phi_M & ~~ & \Phi_3^0 = \phi_P \\
    \Psi_1^0 = -\int_0^\theta \phi_P(z)\,dz & ~~ & \Psi_2^0 = 1 & ~~ & \Psi_3^0 = \phi,
\end{array}
\]
then $\{\Phi_\ell^0\}_{\ell=1}^3$ and $\{\Psi_j^0\}_{j=1}^3$ are biorthogonal bases for the generalized kernels of $\mathcal{A}_0[\phi]$ and $\mathcal{A}_0^\dagger[\phi]$ acting on $ L^{2}_{\rm per}[0,2\pi) $ respectively, i.e., $\LA \Psi_j^0, \Phi_\ell^0\RA = \delta_{j\ell}$. The $\Phi_\ell^0$ and $\Psi_j^0$ satisfy the equations
\[
\mathcal{A}_0[\phi]\Phi_1^0 = 0, \quad \mathcal{A}_0[\phi]\Phi_2^0 = -kc_M\Phi_1^0, \quad \mathcal{A}_0[\phi]\Phi_3^0 = -kc_P\Phi_1^0
\]
and
\[
\mathcal{A}_0^\dagger[\phi]\Psi_2^0 = 0 = \mathcal{A}_0^\dagger[\phi]\Psi_3^0, \quad \mathcal{A}_0^\dagger[\phi]\Psi_1^0 = -b_p\Psi_2^0 - kc_P\Psi_3^0.
\]
\end{proposition}
\begin{proof}
First observe from \Cref{e:kernelelmts} that 
\[
\mathcal{L}[\phi]\{c,\phi\}_{M,P} = -\{c, b\}_{M,P} \in \Span\{1\},
\]
where we have introduced the convenient notation
\[
\{f,g\}_{\alpha,\beta} := 
\left|\begin{array}{cc}
    f_\alpha & f_\beta \\
    g_\alpha & g_\beta
\end{array}\right| 
= f_\alpha g_\beta - f_\beta g_\alpha.
\]
By \Cref{assumption:a4}, $ \dim\ker(\mathcal{L}[\phi]) = 1 $, so $\{c,b\}_{M,P}\neq 0 $ which implies that $\phi' =\Phi_1^0 $ and $ \{c,\phi\}_{M,P}\in\Span\{\Phi_2^0, \Phi_3^0\} $ are linearly independent. \Cref{assumption:a4} additionally implies that the kernel of $\mathcal{A}_0[\phi] = \d_\theta\mathcal{L}[\phi]$ is at most two-dimensional, so we may conclude that
\[
\ker(\mathcal{A}_0[\phi]) = \Span\{\Phi_1^0, \{c,\phi\}_{M,P}\}.
\]
Moreover, \Cref{e:kernelelmts2} implies that $c_M$ and $c_P$ cannot simultaneously vanish. By duality, we have that the kernel of $\mathcal{A}_0^\dagger[\phi]$ is at most two-dimensional, so that \Cref{e:kernelelmts3} implies
\[
\ker(\mathcal{A}_0^\dagger[\phi]) = \Span\{\Psi_2^0,\Psi_3^0\}.
\]
From \Cref{e:innprod1} and the fundamental theorem of calculus we have
\begin{align*}
	\LA \Psi_2^0, \Phi_2^0\RA &=1 = \LA \Psi_3^0, \Phi_3^0\RA\\
	\LA \Psi_2^0, \Phi_3^0\RA &= 0 = \LA \Psi_3^0, \Phi_2^0\RA\\
	\LA \Psi_2^0, \Phi_1^0\RA &= 0 = \LA \Psi_3^0, \Phi_1^0\RA.
\end{align*}
Integration by parts yields
\[
\LA \Psi_1^0, \Phi_1^0\RA = \LA \Phi_3^0,\Psi_3^0\RA =1,
\]
and finally
\[
\LA \Psi_1^0, \Phi_2^0\RA = 0 = \LA \Psi_1^0, \Phi_3^0\RA
\]
by parity.  Consequently, the Fredholm Alternative implies that $\lambda=0$ is an eigenvalue with algebraic multiplicity three and geometric multiplicity two.
\end{proof}

\begin{remark}
Since $-b_P\Psi_2^0 - kc_P\Psi_3^0$ lies in the range of $\mathcal{A}_0^\dagger[\phi]$, the Fredholm Alternative implies that $-b_P\Psi_2^0 - kc_P\Psi_3^0$ is orthogonal to the kernel of $\mathcal{A}[\phi]$.  Clearly, it is orthogonal to $\Phi_1^0$.  We also have that $\{c,\phi\}_{M,P} = c_M\Phi_3^0 - c_P\Phi_2^0 \in \ker(\mathcal{A}_0[\phi])$, hence
\[
0 = \LA -b_P\Psi_2^0 - kc_P\Psi_3^0, c_M\Phi_3^0 - c_P\Phi_2^0\RA = b_Pc_P - kc_Mc_P \Longrightarrow b_P = kc_M ~~\text{or}~~ c_P=0.
\]
We view this as a consequence of \Cref{assumption:a4} on the structure of the generalized kernel of $ \mathcal{A}_{0}[\phi] $.
\end{remark}

We now make some observations regarding $\phi_k$.  Differentiating \Cref{e:profile2} with respect to $k$ yields
\begin{align*}
	\mathcal{L}[\phi]\phi_k = -b_k - (kc)_k\phi + f(\phi) + \mathcal{K}*\phi + k\partial_{k}\int_{\R}K(\xi)\phi(\theta - k\xi)d\xi.
\end{align*}
We justify the passage of $ \partial_{k} $ under the integral since $ K(\xi)\partial_{k}\phi(\theta - k\xi) = -\xi K(\xi)\phi'(\theta-k\xi) \in L^{1}(\R)$, $ \xi K(\xi) \in L^{1}(\R) $ and $ \phi'(\theta-k\xi) < J $ for some constant $ J \in \R $. Hence:
\begin{align*}
	\mathcal{L}[\phi]\phi_k = -b_k - (kc)_k\phi + f(\phi) + \mathcal{K}*\phi - \mathcal{K}_{1}*\phi'
\end{align*}
so that
\begin{equation}\label{e:phik_identity}
\mathcal{A}_0\phi_k = -\left(c - f'(\phi)-\mathcal{K}*\right)\phi' - \mathcal{K}_{1}*\phi'' - kc_k\phi'.
\end{equation}
Furthermore, taking the derivatives of \Cref{e:M} and \Cref{e:P} with respect to $k$ yields
\begin{equation}\label{e:phik_prods}
\LA \Psi_2^0, \phi_k\RA = 0 = \LA \Psi_3^0, \phi_k\RA.
\end{equation}
\subsection{Modulational Stability Calculation}
Now we examine the bifurcation of the triple eigenvalue from $ (\lambda,\tau) = (0,0) $ by taking a Taylor expansion of the Bloch operators $\mathcal{A}_\tau[\phi]$ for $|\tau|\ll 1$.  Recall that
\begin{align*}
	\mathcal{A}_\tau[\phi]v &= e^{-i\tau\theta}\mathcal{A}_{\tau}[\phi]e^{i\tau\theta}v\\
	&= e^{-i\tau\theta}\d_\theta \mathcal{L}[\phi]e^{i\tau\theta}v\\
	&= (\partial_{\theta} + i\tau)e^{-i\tau\theta}\mathcal{L}[\phi]e^{i\tau\theta}v\\
	&= k(\d_\theta + i\tau)(c-f'(\phi))v - k(\d_\theta + i\tau) e^{-i\tau\theta}\mathcal{K}*e^{i\tau\theta}v.
\end{align*}
For the convolution term, note that
\begin{align*}
	e^{-i\tau\theta}\mathcal{K}*e^{i\tau\theta}v &= e^{-i\tau\theta}\int_{\R}K(\xi)v(\theta-k\xi)e^{i\tau(\theta-k\xi)}d\xi\\
	&= \int_{\R}K(\xi)v(\theta-k\xi)e^{-ik\tau\xi}d\xi\\
	&= \int_{\R}K(\xi)v(\theta-k\xi)\left(1 - ik\tau\xi - \frac{1}{2}k^{2}\tau^{2}\xi^{2}\right) + Rd\xi,
\end{align*}
where $ R $ is the remainder due to Taylor's theorem, given by
\begin{align*}
	R &= \int_{\R}\int_{0}^{1}\frac{1}{2}(-ik\tau\xi)^{3}(1-t)^{2}e^{-ikt\tau\xi}K(\xi)v(\theta-k\xi)dtd\xi.
\end{align*}
Since $ \xi^{3}K(\xi)\in L^{1}(\R) $ by \Cref{assumption:a1}, we then have:
\begin{align*}
	|R| &\leq \frac{1}{2}\bigg|\int_{\R}\int_{0}^{1}(-ik\tau\xi)^{3}(1-t)^{2}e^{-ikt\tau\xi}K(\xi)v(\theta-k\xi)dtd\xi\bigg|\\
	&\leq \frac{1}{2}\int_{\R}\int_{0}^{1}\bigg|(k\tau\xi)^{3}(1-t)^{2}K(\xi)v(\theta-k\xi)\bigg|dtd\xi\\
	&\leq \frac{1}{2}(k\tau)^{3}\int_{\R}\int_{0}^{1}\bigg|\xi^{3}K(\xi)v(\theta-k\xi)\bigg|dtd\xi.
\end{align*}
This implies that $ R \in \rm O(\tau^{3})$. Recalling the definitions of $\mathcal{K}_{1}$ and $\mathcal{K}_{2}$ in \Cref{eq:K_1,eq:K_2} respectively, we have that
\begin{align*}
	k(\d_\theta + i\tau) e^{-i\tau\theta}\mathcal{K}*e^{i\tau\theta}v &= k\partial_{\theta}\mathcal{K}*v + (ik\tau)\left(\mathcal{K}* - k\partial_{\theta}\mathcal{K}_{1}*\right)v + \frac{1}{2}(ik\tau)^{2}k\partial_{\theta}\mathcal{K}_{2}*v +  \mathrm{O}(\tau^{3}).
\end{align*}
We define the following operators
\begin{align}
	\mathcal{A}_{0} := \mathcal{A}_{0}[\phi],\quad \mathcal{A}_{1} := c - f'(\phi) - \mathcal{K}* + k\partial_{\theta}\mathcal{K}_{1}*,\quad \mathcal{A}_{2} := -\frac{1}{2}k\partial_{\theta}\mathcal{K}_{2},\label{eq:Adefs}
\end{align}
which allows us to write the expansion:
\begin{align}
	\mathcal{A}_{\tau}[\phi] &= \mathcal{A}_{0} + (ik\tau)\mathcal{A}_{1} + (ik\tau)^{2}\mathcal{A}_{2} + \mathrm{O}(\tau^{3}). \label{eq:Aexpansion}
\end{align}
For $ \abs{\tau} \ll 1 $, $ \mathcal{A}_{\tau}[\phi] $ is a relatively compact perturbation of $ \mathcal{A}_{0}[\phi] $ and analytic in $ \tau $. Hence the eigenvalue $ \lambda = 0 $ of multiplicity $ 3 $ bifurcates into three eigenvalues $ \{\lambda_{j}(\tau)\}_{j = 1}^{3} $ for $ 0 < \abs{\tau} \ll 1 $. The theory of Kato in \cite{K76} allows us to extend analytically the dual bases for the generalized kernels of $ \mathcal{A}_{0}[\phi] $ and $ \mathcal{A}_{0}^{\dagger}[\phi] $ into dual left and right bases, $ \{\Psi_{j}^{\tau}\}_{j = 1}^{3} $ and $  \{\Phi_{j}^{\tau}\}_{j = 1}^{3} $ respectively, for the eigenspaces of $ \mathcal{A}_{\tau}[\phi] $ associated with the eigenvalues $ \{\lambda_{j}(\tau)\}_{j = 1}^{3} $. In particular, we preserve the biorthogonality $ \langle \Psi_{j}^{\tau},\Phi_{\ell}^{\tau}\rangle = \delta_{j\ell} $ for all $ \abs{\tau}\ll 1 $. As in \cite[Theorem 1]{BNR14} and \cite{JP2020}, we examine the action of $ \mathcal{A}_{\tau}[\phi] $ on the total eigenspace, which is determined by the matrix
\[
D_{\tau}:=\left(\LA\Psi_j^\tau,\mathcal{A}_\tau[\phi]\Phi_\ell^\tau\RA\right)_{j,\ell=1}^3,
\]
where we are using the $ L^{2}[0,2\pi) $ inner product such that, for $ d \in \C $, we have
\begin{align*}
	\langle f, d g\rangle = d\langle f,g\rangle = \langle \overline{d} f,g\rangle.
\end{align*}
The matrix $ D_{\tau} - \lambda I $, where $ I $ is the $ 3\times3 $ identity matrix, is singular when $ \lambda = \lambda_{j}(\tau) $. We therefore proceed by computing the leading order entries of the matrix $ D_{\tau} $. With the Taylor expansion
\begin{align*}
	\Phi_{\ell}^{\tau} = \Phi_{\ell}^{0} + (ik\tau)\left(\frac{1}{ik}\partial_{\tau}\Phi_{\ell}^{\tau}\bigg|_{\tau = 0}\right) + (ik\tau)^{2}\left(\frac{1}{2(ik)^{2}}\partial_{\tau}^{2}\Phi_{\ell}^{\tau}\bigg|_{\tau = 0}\right) + \mathrm{O}(\tau^{3}),
\end{align*}
and similarly for $ \Psi_{j}^{\tau} $, we expand the matrix $ D_{\tau} $ about $ \tau = 0 $, so that
\begin{align*}
	D_{\tau} = D_{0} + ik\tau D_{1} + (ik\tau)^{2}D_{2} + \mathrm{O}(\tau^{3}).
\end{align*}
From \Cref{P:gkerper} we have that
\begin{align*}
	D_{0} = \begin{pmatrix}
		0 & -kc_{M} & -kc_{P}\\
		0 & 0 & 0\\
		0 & 0 & 0
	\end{pmatrix}.
\end{align*}
The matrix $ D_{1} $ is given by
\begin{align}
	D_{1} = \left(\left\langle\Psi_{j}^{0}, \frac{1}{ik}\mathcal{A}_{0}[\phi]\partial_{\tau}\Phi_{\ell}^{\tau}\bigg|_{\tau = 0} + \mathcal{A}_{1}[\phi]\Phi_{\ell}^{0}\right\rangle + \left\langle\frac{1}{ik}\partial_{\tau}\Psi_{j}^{\tau}\bigg|_{\tau = 0},\mathcal{A}_{0}[\phi]\Phi_{j}^{0} \right\rangle\right)_{j,\ell = 1}^{3}.\label{eq:D1}
\end{align}
We first compute $ \partial_{\tau}\Phi_{1}^{\tau}\big|_{\tau = 0} $ by noting that:
\begin{align}
	\partial_{\tau}(\mathcal{A}_{\tau}[\phi]\Phi_{1}^{\tau})\big|_{\tau = 0} &= \partial_{\tau}(\lambda_{1}(\tau)\Phi_{1}^{\tau})\big|_{\tau = 0}\nn\\
	\implies \mathcal{A}_{0}\partial_{\tau}\Phi_{1}^{\tau}\big|_{\tau = 0} + ik\mathcal{A}_{1}\Phi_{1}^{0} &= \lambda_{1}'(0)\Phi_{1}^{0} + \lambda_{1}(0)\partial_{\tau}\Phi_{1}^{\tau}\big|_{\tau = 0}.\label{eq:phitauderiv}
\end{align}
Since $ \lambda_{1}(0) = 0 $ by \Cref{P:gkerper}, and also noting that we can rewrite \Cref{e:phik_identity} as
\begin{align}
	\mathcal{A}_{1}\Phi_{1}^{0} = -\mathcal{A}_{0}\phi_{k} - kc_{k}\Phi_{1}^{0},\label{eq:A1Phi1}
\end{align}
we simplify \Cref{eq:phitauderiv} to
\begin{align*}
	\mathcal{A}_{0}\left(\partial_{\tau}\Phi_{1}^{\tau}\big|_{\tau = 0} - ik\phi_{k}\right) = (\lambda_{1}'(0) + ik^{2}c_{k})\Phi_{1}^{0}.
\end{align*}
From our assumptions on the structure of the generalized kernel of $ \mathcal{A}_{0} $ in \Cref{P:gkerper}, we have that $ \partial_{\tau}\Phi_{1}^{\tau}\big|_{\tau = 0} - ik\phi_{k} \in \mathrm{span}\{\Phi_{1}^{0},\Phi_{2}^{0},\Phi_{3}^{0}\} $. If we let
\begin{align*}
	\frac{1}{ik}\partial_{\tau}\Phi_{1}^{\tau}\bigg|_{\tau = 0} = \phi_{k} + \sum_{j = 1}^{3}a_{j}\Phi_{j}^{0}
\end{align*}
for constants $ a_{j} $, we can define
\begin{align*}
	\widetilde{\Phi}_{1}^{\tau} &:= \Phi_{1}^{\tau} - ik\tau\sum_{\ell = 1}^{3}a_{\ell}\Phi_{\ell}^{\tau}\\
	\widetilde{\Psi}_{j}^{\tau} &:= \Psi_{j}^{\tau} - ik\tau \overline{a_{j}}\Psi_{1}^{\tau},
\end{align*}
and we note that
\begin{align}
	\left\langle\widetilde{\Psi}_{j}^{\tau},\widetilde{\Phi}_{\ell}^{\tau} \right\rangle = \delta_{j\ell} + \rm O(\tau^{2}).\label{eq:newbasesortho}
\end{align}
Also, we have that
\begin{align*}
	-kc_{M}a_{2} - kc_{P}a_{3} &= \frac{1}{ik}\lambda_{1}'(0) + kc_{k}.
\end{align*}
We observe that these new bases span the same spaces as the original bases, since
\begin{align*}
	\begin{pmatrix}
		\widetilde{\Phi}_{1}^{\tau}\\\widetilde{\Phi}_{2}^{\tau}\\\widetilde{\Phi}_{3}^{\tau}
	\end{pmatrix}
	&= \begin{pmatrix}
		1 - ik\tau a_{1} & -ik\tau a_{2} & -ik\tau a_{3}\\
		0 & 1 & 0\\
		0 & 0 & 1
	\end{pmatrix}
	\begin{pmatrix}
		\Phi_{1}^{\tau}\\\Phi_{2}^{\tau}\\\Phi_{3}^{\tau}
	\end{pmatrix}\\
	\begin{pmatrix}
		\widetilde{\Psi}_{1}^{\tau}\\\widetilde{\Psi}_{2}^{\tau}\\\widetilde{\Psi}_{3}^{\tau}
	\end{pmatrix}
	&= \begin{pmatrix}
		1 - ik\tau a_{1} & 0 & 0\\
		-ik\tau a_{2} & 1 & 0\\
		-ik\tau a_{3} & 0 & 1
	\end{pmatrix}
	\begin{pmatrix}
		\Psi_{1}^{\tau}\\\Psi_{2}^{\tau}\\\Psi_{3}^{\tau}
	\end{pmatrix},
\end{align*}
where both transformation matrices are invertible provided $ \tau \neq \frac{1}{ika_{1}} $. So we drop the tildes and proceed using the new bases instead. We now have, for example,
\begin{align*}
	\Phi_{1}^{\tau} &= \Phi_{1}^{0} + ik\tau\phi_{k} + \rm O(\tau^{2}),
\end{align*}
so that from \Cref{eq:D1} we compute
\begin{align*}
	(D_{1})_{j1} &= \left\langle\Psi_{j}^{0}, \mathcal{A}_{0}[\phi]\phi_{k} + \mathcal{A}_{1}[\phi]\Phi_{1}^{0}\right\rangle + \left\langle\frac{1}{ik}\partial_{\tau}\Psi_{j}^{\tau}\bigg|_{\tau = 0},\mathcal{A}_{0}[\phi]\Phi_{1}^{0} \right\rangle\\
	&= \left\langle\Psi_{j}^{0}, -kc_{k}\Phi_{1}^{0}\right\rangle\\
	&= -kc_{k}\delta_{j1},
\end{align*}
where we have used \Cref{P:gkerper,eq:A1Phi1}. We note also from \Cref{eq:newbasesortho} that:
\begin{align*}
	\partial_{\tau}\langle\Psi_{j}^{\tau},\Phi_{\ell}^{\tau} \rangle\big|_{\tau = 0} &= 0\\
	\implies \left\langle\partial_{\tau}\Psi_{j}^{\tau}\big|_{\tau = 0}, \Phi_{\ell}^{0} \right\rangle &= - \left \langle\Psi_{j}^{0},\partial_{\tau}\Phi_{\ell}^{\tau}\big|_{\tau = 0} \right \rangle.
\end{align*}
Introducing the convenient notation $ c_{\ell} = c_{M}, c_{P} $ for $ \ell = 2,3 $ respectively, we then have
\begin{align*}
	(D_{1})_{j\ell} &= \left\langle\Psi_{j}^{0}, \frac{1}{ik}\mathcal{A}_{0}[\phi]\partial_{\tau}\Phi_{\ell}^{\tau}\bigg|_{\tau = 0} + \mathcal{A}_{1}[\phi]\Phi_{\ell}^{0}\right\rangle + \left\langle\frac{1}{ik}\partial_{\tau}\Psi_{j}^{\tau}\bigg|_{\tau = 0},\mathcal{A}_{0}[\phi]\Phi_{\ell}^{0} \right\rangle\\
	&= \frac{1}{ik}\left\langle\mathcal{A}_{0}^{\dagger}[\phi]\Psi_{j}^{0}, \partial_{\tau}\Phi_{\ell}^{\tau}\big|_{\tau = 0}\right\rangle + \langle\Psi_{j}^{0},  \mathcal{A}_{1}[\phi]\Phi_{\ell}^{0}\rangle - \frac{1}{ik}\left\langle\partial_{\tau}\Psi_{j}^{\tau}\big|_{\tau = 0},-kc_{\ell}\Phi_{1}^{0} \right\rangle\\
	&= \langle\Psi_{j}^{0},  \mathcal{A}_{1}[\phi]\Phi_{\ell}^{0}\rangle + \frac{1}{ik}\left\langle\Psi_{j}^{0},-kc_{\ell}\partial_{\tau}\Phi_{1}^{\tau}\big|_{\tau = 0} \right\rangle\\
	&= \langle\Psi_{j}^{0},  \mathcal{A}_{1}[\phi]\Phi_{\ell}^{0}\rangle + \left\langle\Psi_{j}^{0},-kc_{\ell}\phi_{k}\right\rangle\\
	&= \langle\Psi_{j}^{0},  \mathcal{A}_{1}[\phi]\Phi_{\ell}^{0}\rangle,
\end{align*}
recalling \Cref{e:phik_prods}. So we have
\begin{align*}
	D_{1} &= \begin{pmatrix}
		-kc_{k} & * & *\\
		0 & \langle\Psi_{2}^{0},  \mathcal{A}_{1}[\phi]\Phi_{2}^{0}\rangle & \langle\Psi_{2}^{0},  \mathcal{A}_{1}[\phi]\Phi_{3}^{0}\rangle\\
		0 & \langle\Psi_{3}^{0},  \mathcal{A}_{1}[\phi]\Phi_{2}^{0}\rangle & \langle\Psi_{3}^{0},  \mathcal{A}_{1}[\phi]\Phi_{3}^{0}\rangle
	\end{pmatrix},
\end{align*}
where $ * $ represents an entry for which we have already calculated leading order behaviour in $ D_{0} $. Finally, we compute the leading order behaviour of the $ (2,1) $ and $ (3,1) $ entries of $ D $:
\begin{align*}
	(D_{2})_{j1} &= \left\langle\Psi_{j}^{0},-\frac{1}{2k^{2}}\mathcal{A}_{0}[\varphi]\partial_{\tau}^{2}\Phi_{1}^{\tau}\bigg\lvert_{\tau = 0} + \frac{1}{ik}\mathcal{A}_{1}[\varphi]\partial_{\tau}\Phi_{1}^{\tau}\big\lvert_{\tau = 0} + \mathcal{A}_{2}[\varphi]\Phi_{1}^{0}\right\rangle\\
	&+ \left\langle\frac{1}{ik}\partial_{\tau}\Psi_{j}^{\tau}\bigg|_{\tau = 0},\frac{1}{ik}\mathcal{A}_{0}[\varphi]\partial_{\tau}\Phi_{1}^{\tau}\bigg|_{\tau = 0} + \mathcal{A}_{1}[\varphi]\Phi_{1}^{0}\right\rangle + \left\langle -\frac{1}{2k^{2}}\partial_{\tau}^{2}\Psi_{l}^{\tau}\bigg\lvert_{\tau = 0},\mathcal{A}_{0}[\varphi]\Phi_{1}^{0}\right \rangle\\
	&= \langle \Psi_{j}^{0},\mathcal{A}_{1}\phi_{k} + \mathcal{A}_{2}\Phi_{1}^{0}\rangle - \frac{1}{ik}\left\langle\partial_{\tau}\Psi_{j}^{\tau}\big|_{\tau = 0},\mathcal{A}_{0}\phi_{k} + \mathcal{A}_{1}[\varphi]\Phi_{1}^{0}\right\rangle\\
	&= \langle \Psi_{j}^{0},\mathcal{A}_{1}\phi_{k} + \mathcal{A}_{2}\Phi_{1}^{0}\rangle - \frac{1}{ik}\left\langle\partial_{\tau}\Psi_{j}^{\tau}\big|_{\tau = 0},-kc_{k}\Phi_{1}^{0}\right\rangle\\
	&= \langle \Psi_{j}^{0},\mathcal{A}_{1}\phi_{k} + \mathcal{A}_{2}\Phi_{1}^{0}\rangle.
\end{align*}
As in \cite{BNR14}, we have shown that the limit as $ \tau \rightarrow 0 $ of
\begin{align*}
	\widehat{D}_{\tau} := \begin{pmatrix}
		\frac{1}{ik\tau}\langle \Psi_{1}^\tau,\mathcal{A}_\tau[\phi]\Phi_{1}^{\tau}\rangle & \langle \Psi_{1}^\tau,\mathcal{A}_\tau[\phi]\Phi_{2}^{\tau}\rangle & \langle \Psi_{1}^\tau,\mathcal{A}_\tau[\phi]\Phi_{3}^{\tau}\rangle\\
		-\frac{1}{k^{2}\tau^{2}}\langle \Psi_{2}^\tau,\mathcal{A}_\tau[\phi]\Phi_{1}^{\tau}\rangle & \frac{1}{ik\tau}\langle \Psi_{2}^\tau,\mathcal{A}_\tau[\phi]\Phi_{2}^{\tau}\rangle & \frac{1}{ik\tau}\langle \Psi_{2}^\tau,\mathcal{A}_\tau[\phi]\Phi_{3}^{\tau}\rangle\\
		-\frac{1}{k^{2}\tau^{2}}\langle \Psi_{3}^\tau,\mathcal{A}_\tau[\phi]\Phi_{1}^{\tau}\rangle & \frac{1}{ik\tau}\langle \Psi_{3}^\tau,\mathcal{A}_\tau[\phi]\Phi_{2}^{\tau}\rangle & \frac{1}{ik\tau}\langle \Psi_{3}^\tau,\mathcal{A}_\tau[\phi]\Phi_{3}^{\tau}\rangle
	\end{pmatrix}
\end{align*}
is
\begin{align}
	\widehat{D}_{0} = \begin{pmatrix}
		-kc_{k} & -kc_{M} & -kc_{P}\\
		\langle \Psi_{2}^{0},\mathcal{A}_{1}\phi_{k} + \mathcal{A}_{2}\Phi_{1}^{0}\rangle & \langle\Psi_{2}^{0},  \mathcal{A}_{1}[\phi]\Phi_{2}^{0}\rangle & \langle\Psi_{2}^{0},  \mathcal{A}_{1}[\phi]\Phi_{3}^{0}\rangle\\
		\langle \Psi_{3}^{0},\mathcal{A}_{1}\phi_{k} + \mathcal{A}_{2}\Phi_{1}^{0}\rangle & \langle\Psi_{3}^{0},  \mathcal{A}_{1}[\phi]\Phi_{2}^{0}\rangle & \langle\Psi_{3}^{0},  \mathcal{A}_{1}[\phi]\Phi_{3}^{0}\rangle
	\end{pmatrix}.
\end{align}
We note that
\begin{align*}
	\widehat{D}_{\tau} = \frac{1}{ik\tau}S(\tau)D_{\tau}S(\tau)^{-1},
\end{align*}
where
\begin{align*}
	S(\tau):= \begin{pmatrix}
		ik\tau & 0 & 0\\
		0 & 1 & 0\\
		0 & 0 & 1
	\end{pmatrix},
\end{align*}
and hence:
\begin{align}
	\det(D_{\tau} - \lambda I) = (ik\tau)^{3}\det(\widehat{D}_{\tau} - \frac{\lambda}{ik\tau}I). \label{eq:spectra}
\end{align}
Substituting one of the eigenvalue branches $ \lambda = \lambda_{j}(\tau) $ into \Cref{eq:spectra} and taking the limit as $ \tau \rightarrow 0 $ (noting that the determinant is multilinear), we see that
\begin{align*}
	\det(\widehat{D}_{0} - \frac{\lambda_{j}'(0)}{ik}I) = 0,
\end{align*}
so we conclude that a complex eigenvalue of $ \widehat{D}_{0} $ implies the existence of an eigenvalue $ \lambda_{j}(\tau) $ bifurcating from zero into the right (or left) half-plane. We claim that
\begin{align}
	{\bf D}(u_{0}) = \widehat{D}_{0} - cI,\label{eq:connection}
\end{align}
where $ {\bf D}(u_{0}) $ is the coefficient matrix from the Whitham theory defined in \Cref{eq:WhithamMatrix}. \Cref{eq:connection} implies that the eigenvalues of $ {\bf D}(u_{0}) $ and $ \widehat{D}_{0} $ have equal imaginary parts, which is sufficient to prove \Cref{T:main}.
\subsection{Proving \Cref{T:main}}
We now prove \Cref{T:main} by verifying \Cref{eq:connection}. The entries in the first rows of the respective matrices are equal. For the second row, we first note from \Cref{eq:Adefs} that for $ v \in L^{2}_{\rm per}[0,2\pi)$
\begin{align*}
	-\langle1,f'(\phi)v + \mathcal{K}*v \rangle &= \left\langle1,\left(c - f'(\phi) - \mathcal{K}* + k\partial_{\theta}\mathcal{K}_{1} - \frac{1}{2}k\partial_{\theta}\mathcal{K}_{2}*\right)v \right\rangle - c\langle1,v\rangle\\
	&= \langle\Psi_{2}^{0},(\mathcal{A}_{1} + \mathcal{A}_{2} - c)v\rangle,
\end{align*}
where Fubini's theorem and the commutativity of $ \partial_{\theta} $ and $ \mathcal{K}_{1},\mathcal{K}_{2} $ (a consequence of \Cref{L:Kprop}) ensure that the terms involving $ \partial_{\theta} $ vanish. In particular, $ \langle \Psi_{2}^{0},\mathcal{A}_{2}v \rangle = 0 $, so we drop this term altogether. For $ v = \Phi_{\ell}^{0} $ with $ \ell = 2,3 $, we have from \Cref{e:phik_prods,e:innprod1} that
\begin{align*}
	-\langle1,f'(\phi)\Phi_{\ell}^{0} + \mathcal{K}*\Phi_{\ell}^{0} \rangle &= \langle\Psi_{2}^{0},\mathcal{A}_{1}\Phi_{\ell}^{0}\rangle - c\delta_{2\ell},
\end{align*}
and for $ v = \phi_{k} $ we have
\begin{align*}
	-\langle1,f'(\phi)\phi_{k} + \mathcal{K}*\phi_{k} \rangle &= \langle\Psi_{2}^{0},\mathcal{A}_{1}\phi_{k}\rangle\\
	&= \langle\Psi_{2}^{0},\mathcal{A}_{1}\phi_{k} + \mathcal{A}_{2}\Phi_{1}^{0}\rangle.
\end{align*}
This shows that the second rows of $\widehat{D}_0 - cI$ and ${\bf D}(\phi)$ are identical. Finally, for the second and third entry of the third row we have:
\begin{align*}
	-\langle \phi,f'(\phi)\Phi_{\ell}^{0} + \mathcal{K}*\Phi_{\ell}^{0} - k\mathcal{K}_{1}*\partial_{\theta}\Phi_{\ell}^{0}\rangle &= \langle \Psi_{3}^{0},(\mathcal{A}_{1} - c)\Phi_{\ell}^{0}\rangle\\
	&= \langle\Psi_{3}^{0},\mathcal{A}_{1}\Phi_{\ell}^{0}\rangle - c\delta_{3\ell},
\end{align*}
and for the first entry:
\begin{align*}
	-\left\langle \phi,f'(\phi)\phi_{k} + \mathcal{K}*\phi_{k} - k\mathcal{K}_{1}*\phi_{k}' - \mathcal{K}_{1}*\phi' + \frac{1}{2}k\mathcal{K}_{2}*\phi'' \right\rangle &= \langle\Psi_{3}^{0},(\mathcal{A}_{1} - c)\phi_{k} + \mathcal{A}_{2}\Phi_{1}^{0} \rangle - \langle\phi,\mathcal{K}_{1}*\phi'\rangle.
\end{align*}
The $ -c\phi_{k} $ vanishes due to \Cref{e:phik_prods}, and for the last term we note from the symmetry of $ \mathcal{K}_{1} $ and integration by parts
\begin{align*}
	\langle\phi,\mathcal{K}_{1}*\phi'\rangle &= -\langle\partial_{\theta}\mathcal{K}_{1}\phi,\phi\rangle\\
	&= -\langle\phi,\mathcal{K}_{1}*\phi'\rangle,
\end{align*}
where taking the conjugate is unnecessary since $ \phi $ and $ \xi K(\xi) $ are real by \Cref{assumption:a1,assumption:a3}. This implies that $ \langle\phi,\mathcal{K}_{1}*\phi'\rangle = 0 $, and so we have
\begin{align*}
	-\left\langle \phi,f'(\phi)\phi_{k} + \mathcal{K}*\phi_{k} - k\mathcal{K}_{1}*\phi_{k}' - \mathcal{K}_{1}*\phi' + \frac{1}{2}k\mathcal{K}_{2}*\phi'' \right\rangle &= \langle\Psi_{3}^{0},\mathcal{A}_{1}\phi_{k} + \mathcal{A}_{2}\Phi_{1}^{0} \rangle,
\end{align*}
completing the verification of \Cref{eq:connection}. Thus, we have shown that weak hyperbolicity is a necessary condition for stability and ellipticity is a sufficient condition for instability. If $ {\bf D}(u_{0}) $ is strictly hyperbolic, i.e., it has three distinct real eigenvalues $ \{\mu_{j}\}_{j = 1}^{3} $, then $ \widehat{D}_{0} $ has three distinct eigenvalues with leading order behaviour:
\begin{align*}
	\lambda_{j}(\tau) = ik\tau(\mu_{j} + c) + \rm O(\tau^2).
\end{align*}
Since the linearized operator $ \mathcal{A}_{0}[\phi] = \partial_{\theta}\mathcal{L} $ has a Hamiltonian structure, its spectrum is symmetric about the real and imaginary axes. Following the reasoning of \cite{JZ2010}, if one of the $ \lambda_{j} $ bifurcates into the right half-plane, then we would have WLOG $ \lambda_{j} = -\overline{\lambda_{\ell}} $. This implies $ \mu_{j} = \mu_{\ell} $, which contradicts the fact that $ {\bf D}(u_{0}) $ is strictly hyperbolic. We conclude that strict hyperbolicity of the leading order Whitham modulation equations is a sufficient condition for the modulational stability of the underlying wave. This completes the proof of \Cref{T:main}. 
\begin{appendices}\crefalias{section}{appsec}
	\section{Derivation of the Whitham modulation equations}\label[appendix]{appendix:aA}
	The derivation of the Whitham modulation equations in \cite{BIHS2021} relies on the fact that  $ \Omega(q) $ is analytic, where
	\[
	\widehat{K}(q) = \frac{\Omega(q)}{q}.
	\]
	Also, the functions $ u_{0},u_{1},\dots $ from the expansion $ u(x,t) = u_{0}(\theta,X,T) + \epsilon u_{1}(\theta,X,T) + \mathrm{O}(\epsilon^{2}) $ are assumed to have Fourier transforms in a distributional sense. In order to deal with the prospect of exponential growth in the slow spatial and temporal scales, which can arise when the Whitham modulation equations are elliptic, the usual extension of the Fourier transform to the space of tempered distributions is invalid. Instead, test functions must be taken from a suitable Gelfand-Shilov space, and we believe this technical extension lies outside the scope of our paper (see \cite{CCK1996}, for example). Instead of using Fourier analysis, we re-derive the Whitham modulation equations by taking a Taylor expansion within the convolution, which we justify using a different set of assumptions to those in \cite{BIHS2021}. We start with the usual expansion of a solution to the PDE
	\[
	u(x,t) = u_{0}\left(\frac{1}{\epsilon}\psi(X,T),X,T\right) + \epsilon u_{1}\left(\frac{1}{\epsilon}\psi(X,T),X,T\right) + \mathrm{O}(\epsilon^{2}),
	\]
	with $ \psi $ chosen such that the $ u_{j} $ are periodic in the first variable. We define
	\[ 
	\theta:=\frac{1}{\epsilon}\psi(X,T),k:=\psi_{X}(X,T),\omega:=\psi_{T}(X,T),c:=\frac{\omega}{k},
	\]
	and following \cite{L66,Whitham1970,Whitham1999} we consider $ \theta $ to be an independent variable at certain points of the analysis in order to ensure a uniform expansion in $ \epsilon $. Once this ansatz is substituted into the PDE, we have:
	\begin{align*}
		-kcu_{0\theta} + f(u_{0} + \epsilon u_{1} + \dots)_{x} + \mathcal{K}*(u_{0} + \epsilon u_{1} + \dots)_{x} + \mathrm{O}(\epsilon) &= 0,
	\end{align*}
	and upon further simplification:
	\begin{equation}\label{eq:WExpand}
	\begin{aligned}
		&-kcu_{0\theta} + kf'(u_{0})u_{0\theta} + \epsilon(u_{0T} + f'(u_{0})u_{0X} + f''(u_{0})u_{1}u_{0\theta} -kcu_{1\theta} + kf'(u_{0})u_{1\theta})\\
		&+ \mathcal{K}*(u_{0} + \epsilon u_{1} + \dots)_{x}  + \mathrm{O}(\epsilon^{2}) = 0. 
	\end{aligned}
	\end{equation}
	Let $ X,T $ be fixed. The expansion of the convolution requires particular care; the first term is:
	\begin{align*}
		K*u_{0x} &= \int_{\R}K(\xi)\partial_{x}u_{0}\left(\frac{1}{\epsilon}\psi(X-\epsilon\xi,T),X-\epsilon\xi,T\right)d\xi\\
		&= \int_{\R}K(\xi)\bigg[k(X-\epsilon\xi,T)u_{0\theta}\left(\frac{1}{\epsilon}\psi(X-\epsilon\xi,T),X-\epsilon\xi,T\right) + \epsilon u_{0X}\left(\frac{1}{\epsilon}\psi(X-\epsilon\xi,T),X-\epsilon\xi,T\right)\bigg]d\xi,
	\end{align*}
	where we abuse notation in that $ u_{0X} $ denotes differentiation only in the second variable. We have also used the fact that $ k(X,T) = \psi_{X}(X,T) $. We need to take two Taylor expansions: firstly note that
	\begin{align*}
		\frac{1}{\epsilon}\psi(X-\epsilon\xi,T) &= \frac{1}{\epsilon}\psi(X,T) - \epsilon\xi\frac{1}{\epsilon}\psi_{X}(X,T) + \frac{1}{2}\epsilon^{2}\xi^{2}\frac{1}{\epsilon}\psi(X-\epsilon\xi,T) + \epsilon^{2}\xi^{2}\frac{1}{\epsilon}R_{\psi}(X-\epsilon\xi,T)\\
		&= \theta - k\xi + \frac{1}{2}\epsilon\xi^{2}k_{X} + \epsilon\xi^{2}R_{\psi}(X-\epsilon\xi,T),
	\end{align*}
	where $ R_{\psi} $ is the Peano form of the remainder due to Taylor's theorem. In particular, we have
	\[
	\lim_{\epsilon\xi \rightarrow 0}R_{\psi}(X-\epsilon\xi,T) = 0.\]
	Note that this also implies
	\begin{align*}
		k(X-\epsilon\xi,T) &= k - \epsilon\xi k_{X} + \frac{1}{2}\epsilon^{2}\xi^{2}k_{XX} + \epsilon^{2}\xi^{2}R_{\psi X}(X-\epsilon\xi,T),
	\end{align*}
	where the terms involving $ k $ without brackets are evaluated at $ (X,T) $. Next we invoke Taylor's theorem for the function $ u_{0} $:
	\begin{align*}
		&u_{0}\left(\frac{1}{\epsilon}\psi(X-\epsilon\xi,T),X-\epsilon\xi,T\right)\\
		&= u_{0}\left(\theta - k\xi + \frac{1}{2}\epsilon\xi^{2}k_{X} + \epsilon\xi^{2}R_{\psi}(X-\epsilon\xi,T),X-\epsilon\xi,T\right)\\
		&= u_{0}\left(\theta - k\xi,X,T\right) + \left(\frac{1}{2}\epsilon\xi^{2}k_{X} + \epsilon\xi^{2}R_{\psi}(X-\epsilon\xi,T)\right)u_{0\theta}\left(\theta - k\xi ,X,T\right) - \epsilon\xi u_{0X}\left(\theta - k\xi ,X,T\right)\\
		&+ \epsilon\left(\frac{1}{2}\xi^{2}k_{X} + \xi^{2}R_{\psi}(X-\epsilon\xi,T)\right)R_{1}\left(\theta - k\xi + \frac{1}{2}\epsilon\xi^{2}k_{X} + \epsilon\xi^{2}R_{\psi}(X-\epsilon\xi,T),X-\epsilon\xi,T\right)\\
		&+ \epsilon\xi R_{2}\left(\theta - k\xi + \frac{1}{2}\epsilon\xi^{2}k_{X} + \epsilon\xi^{2}R_{\psi}(X-\epsilon\xi,T),X-\epsilon\xi,T\right).
	\end{align*}
	Taylor's theorem requires that
	\begin{align*}
		\lim_{\left(\frac{1}{2}\epsilon\xi^{2}k_{X} + \epsilon\xi^{2}R_{\psi}(X-\epsilon\xi,T),\epsilon\xi\right)\rightarrow (0,0)}R_{i}\left(\theta - k\xi + \frac{1}{2}\epsilon\xi^{2}k_{X} + \epsilon\xi^{2}R_{\psi}(X-\epsilon\xi,T),X - \epsilon\xi,T\right) = 0
	\end{align*}
	for $ i = 1,2 $. We now define the notation
	\begin{align*}
		\mathcal{K}*_{\theta}\varphi := \int_{\R}K(\xi)\varphi(\theta - k\xi)d\xi,
	\end{align*}
	in order to recapture the convolution term in the travelling wave ODE \Cref{e:profile}. We use this different notation only in the context of Taylor expanding the convolution $ \mathcal{K}*u $, taken over $ x $, in order to highlight the fact that $ \mathcal{K}*_{\theta}u $ is taken over $ \theta $. We also use the same notation as in \Cref{eq:K_1,eq:K_2}
	\begin{align*}
		\mathcal{K}_{n}*_{\theta}\varphi := \int_{\R}\xi^{n}K(\xi)\varphi(\theta - k\xi)d\xi.
	\end{align*}
	The expansion of the convolution is ostensibly
	\begin{align*}
		&\mathcal{K}*u_{x} = \mathcal{K}*_{\theta}ku_{0\theta}\\
		&+ \epsilon\bigg[\mathcal{K}*_{\theta}(u_{0X} + ku_{1\theta}) - \mathcal{K}_{1}*_{\theta}\left(k_{X}u_{0\theta} + ku_{0X\theta} - kR_{2\theta}\right) + \mathcal{K}_{2}*_{\theta}\left(\left(\frac{1}{2}k_{X} + R_{\psi}\right)ku_{0\theta\theta} + \left(\frac{1}{2}k_{X} + R_{\psi}\right)kR_{1\theta}\right)\bigg]\\
		&+ \mathrm{O}(\epsilon^{2}).
	\end{align*}
	We have successfully decomposed the convolution operator $ \mathcal{K}* $ over the spatial variable into a series of convolutions carried out with respect to $ \theta $, which recovers the sense of convolution in the travelling wave profile \Cref{e:profile}. However, we must justify why the remainder terms vanish from the expansion to $ {\rm O}(\epsilon) $. Firstly, we are interested in the modulation expansion of a solution satisfying \Cref{e:profile}, so we restrict $ u_{0},u_{1},\dots $ and $ \psi(X,T) $ so that the convolution terms involving remainders satisfy
	\begin{align*}
		k\mathcal{K}_{1}R_{2\theta} + k\mathcal{K}_{2}\left(R_{\psi}u_{0\theta\theta} + \left(\frac{1}{2}k_{X} + R_{\psi}\right)R_{1\theta}\right) &\notin \mathrm{O}\left(\frac{1}{\epsilon}\right),
	\end{align*}
	and so on for higher order terms. Furthermore, we suppose that these remainder terms do not affect the overall expansion to $ \mathrm{O}(\epsilon) $, which we will justify in \Cref{lemma:app1}. To prove that our expansion coincides with the result in \cite{BIHS2021}, first we suppose
	\begin{align*}
		u_{0} = \sum_{n=-\infty}^{\infty}u_{0n}(X,T)e^{in\theta}.
	\end{align*}
	Term-by-term differentiation, under suitable assumptions on the continuity of $ u_{0} $ and convergence of the series, allows us to conclude that
	\begin{align*}
		u_{0\theta} &= \sum_{n=-\infty}^{\infty}inu_{0n}(X,T)e^{in\theta}\\
		u_{0X} &= \sum_{n=-\infty}^{\infty}u_{0nX}(X,T)e^{in\theta}\\
		u_{0\theta\theta} &= \sum_{n=-\infty}^{\infty}-n^{2}u_{0n}(X,T)e^{in\theta}.
	\end{align*}
	The function $ k\mathcal{K}*_{\theta}u_{0\theta} $ is periodic and continuous in $ \theta $, so the $ n $th Fourier coefficient is:
	\begin{align*}
		(k\mathcal{K}*_{\theta}u_{0\theta})_{n} &= \int_{-\pi}^{\pi}\int_{\R}kK(\xi)u_{0\theta}(\theta - k\xi,X,T)e^{-in\theta}d\xi d\theta\\
		&= \int_{\R}inkK(\xi)u_{0n}(X,T)e^{-ink\xi}d\xi\\
		&= i\Omega(nk)u_{0n}(X,T).
	\end{align*}
	Similarly for the other convolution terms, we have:
	\begin{align*}
		&\bigg[\mathcal{K}*_{\theta}u_{0X} - \mathcal{K}_{1}*_{\theta}\left(k_{X}u_{0\theta} + ku_{0X\theta}\right) + \mathcal{K}_{2}*_{\theta}\frac{1}{2}k_{X}ku_{0\theta\theta}\bigg]_{n}\\
		&= \frac{\Omega(nk)}{nk}u_{0nX} - i\left(\frac{\Omega'(nk)}{nk} - \frac{\Omega(nk)}{n^{2}k^{2}}\right)\left(ink_{X}u_{0n} + inku_{0nX}\right) + \frac{1}{2}n^{2}k_{X}k\left(\frac{\Omega''(nk)}{nk} - 2\frac{\Omega'(nk)}{n^{2}k^{2}} + 2\frac{\Omega(nk)}{n^{3}k^{3}}\right)u_{0n}\\
		&= \Omega'(nk)u_{0nX} + \frac{1}{2}nk_{X}\Omega''(nk)u_{0n}\\
		&= \frac{1}{2}\left(\Omega'(nk)u_{0nX} + \partial_{X}(\Omega'(nk)u_{0n})\right).
	\end{align*}
	This agrees with \cite[Lemma 1]{BIHS2021}, where their convolution notation is equivalent to our $ *_{\theta} $. Returning to the derivation of the Whitham equations, we substite the expansion of the convolution operator into \Cref{eq:WExpand}, which yields
	\begin{align*}
		&-kcu_{0\theta} + kf'(u_{0})u_{0\theta} + k\mathcal{K}*_{\theta}u_{0\theta} + \epsilon\bigg(u_{0T} + f'(u_{0})u_{0X} + \mathcal{K}*_{\theta}u_{0X} - \mathcal{K}_{1}*_{\theta}(k_{X}u_{0\theta} + ku_{0X\theta}) + \frac{1}{2}kk_{X}\mathcal{K}_{2}*_{\theta}u_{0\theta\theta}\\
		&+ f''(u_{0})u_{1}u_{0\theta} -kcu_{1\theta} + kf'(u_{0})u_{1\theta} + k\mathcal{K}*_{\theta}u_{1\theta}\bigg) + \mathrm{O}(\epsilon^{2}) = 0.
	\end{align*}
	We identify the travelling wave ODE at $ \mathrm{O}(1) $, so for a fixed $ (X,T) $ we take $ u_{0} $ to be a periodic travelling wave solution of \Cref{e:Whitham}. At $ \mathrm{O}(\epsilon) $, we have:
	\begin{align*}
		-\mathcal{A}[u_{0}]u_{1} = u_{0T} + f'(u_{0})u_{0X} + \mathcal{K}*_{\theta}u_{0X} - \mathcal{K}^{(1)}*_{\theta}(k_{X}u_{0\theta} + ku_{0X\theta}) + \frac{1}{2}kk_{X}\mathcal{K}^{(2)}u_{0\theta\theta}.
	\end{align*}
	The Fredholm alternative requires that
	\begin{align}
		\left(u_{0T} + f'(u_{0})u_{0X} + \mathcal{K}*_{\theta}u_{0X} - \mathcal{K}^{(1)}*_{\theta}(k_{X}u_{0\theta} + ku_{0X\theta}) + \frac{1}{2}kk_{X}\mathcal{K}^{(2)}u_{0\theta\theta}\right)\perp\ker(\mathcal{A}^{\dagger}). \label{e:fredholmwhitham}
	\end{align}
	Recall from \Cref{P:gkerper} that  $ \ker(\mathcal{A}^{\dagger}) = \Span\{1,u_{0}\} $. Hence we derive the Whitham modulation equations from these solvability conditions by taking the $ L^{2}_{\rm per}[0,2\pi) $ inner product of the kernel elements with the expression in \Cref{e:fredholmwhitham}. For the kernel element $ 1 $ we have:
	\begin{align*}
		\langle 1,u_{0T} + f'(u_{0})u_{0X} + \mathcal{K}*_{\theta}u_{0X}\rangle &= 0,
	\end{align*}
	where the convolution terms with $ \theta $-derivatives vanish by Fubini's theorem and the periodicity of $ u_{0} $. If we assume that $ \partial_{T} $ commutes with the inner product, we can recognize that $ M_{T} = \langle 1, u_{0} \rangle_{T} $ and hence this equation reduces to
	\begin{align*}
		M_{T} &= -\langle 1, f(u_{0})_{X} + \mathcal{K}*_{\theta}u_{0X} \rangle.
	\end{align*}
	Making the further assumption that $ \Omega'(0) = 0 $ and identifying $ u_{0} $ as the limiting sum of its Fourier series in $ \theta $ with fixed $ (X,T) $, this equation is identical to the respective modulation equation in \cite{BIHS2021}. This restriction is unnecessary, and we continue our analysis without it. The second modulation equation follows similarly
	\begin{align*}
		P_{T} &= -\langle u_{0},f'(u_{0})u_{X} \rangle - \langle u_{0},\mathcal{K}*_{\theta}u_{0X} - \mathcal{K}_{1}*_{\theta}(k_{X}u_{0\theta} + ku_{0X\theta}) + \frac{1}{2}kk_{X}\mathcal{K}_{2}u_{0\theta\theta} \rangle.
	\end{align*}
	These equations are equivalent to \Cref{e:whitmodsystem}, and one can arrive at the reduced system \Cref{eq:WhithamMatrix} without considering the Fourier transform at all.
	\par To justify the assumption that we may neglect all the convolution terms involving remainders at $ \rm O(\epsilon) $, we note that the formal procedure for determining modulational instability from the first order Whitham modulation equations involves linearizing the equations around a constant solution \cite{Whitham1999,JZ2010,JMMP2014,JP2020,CM2021}. We may restrict the class of functions $ u(\theta,X,T) $ so that $ u $ is bounded in $ X,T $ in order to derive the modulation equations to first order. However, the following lemma outlines a weaker set of assumptions which also yield the desired vanishing of remainder terms at $ \mathrm{O}(\epsilon) $.
	\begin{lemma}\label{lemma:app1}
		Suppose a function $ h\left(X-\epsilon\xi,T\right) $ is increasing in its first variable for $ \abs{\xi} > \xi_{h}(\epsilon) $ for each $ |\epsilon| \leq 1 $. Further, $ h $ satisfies:
		\begin{align*}
			\lim_{\epsilon\rightarrow 0}h(X-\epsilon\xi,T) = 0,
		\end{align*}
		where the limit is understood in a pointwise sense.	Also suppose that $ \xi^{j}K(\xi)y(\theta - k\xi,X,T)h\left(X-\epsilon\xi,T\right) \in L^{1}(\R) $ for each $ |\epsilon| \leq 1 $ and for an integer $ 0\leq j \leq 2 $. Here, $ y $ is $ 2\pi $-periodic and continuous in $ \theta $. Then the limit as $ \epsilon\rightarrow 0 $ commutes with the integral $ \mathcal{K}_{j}*yh $, and in particular we have:
		\begin{align*}
			\lim_{\epsilon \rightarrow 0}\mathcal{K}_{j}*yh = 0.
		\end{align*}
	\end{lemma}
	\begin{proof}
		If $ h(X-\epsilon\xi,T) $ is bounded in $ \xi $, then for fixed $ X,T $ we instantly have a dominating function $ g = \sup_{\xi}h < \infty $. Proceeding under the assumption that $ h(X-\epsilon\xi,T) $ is an increasing, unbounded function for $ \abs{\xi} > \xi_{h}(\epsilon) $ for each $ |\epsilon| \leq 1 $, we note that for the functions
		\begin{align*}
			f_{n}(\xi,X,T) &:=h(X-\epsilon_{n}\xi,T)\\
			\lim_{n\rightarrow \infty}\epsilon_{n} &= 0,
		\end{align*}
		we have
		\begin{align*}
			f_{1}(\epsilon_{n}\xi,X,T) = f_{n}(\xi,X,T).
		\end{align*}
		Without loss of generality we let $ \epsilon_{1} = 1 $, so that the function $ f_{1} $ is assumed to be increasing for $ \abs{\xi} > \xi_{h}(1) $. Note that $ \xi_{h}(\epsilon_{n}) = \frac{\xi_{h}(1)}{\epsilon_{n}} $. A dominating function for $ f_{1} $ is:
		\begin{align*}
			g(\xi,X,T) &:= m + \abs{f_{1}(\xi,X,T)},\\
			m &:= \max_{\abs{\xi} < \xi_{h}(1)}\abs{f_{1}(\xi,X,T)}.
		\end{align*}
		This function dominates all $ f_{n} $, since:
		\begin{align*}
			\abs{f_{n}} \leq \begin{cases}
				m \text{\; when $ \abs{\xi} <  \frac{\xi_{h}(1)}{\epsilon_{n}}$}\\
				f_{1}(\xi,X,T) \text{\; otherwise}
			\end{cases}.
		\end{align*}
		Finally, when taken as part of the convolution, we have
		\begin{align*}
			\norm{\mathcal{K}_{j}*yg}_{L^{1}} &= \int_{\R}\abs{\xi^{j}K(\xi)y(\theta - k\xi,X,T)(M + \abs{h(X-\xi,T)})}d\xi\\
			&\leq m\norm{\mathcal{K}_{j}*_{\theta}y}_{L^{1}} + \norm{\mathcal{K}_{j}*_{\theta}yh}_{L^{1}}\\
			&< \infty,
		\end{align*}
		where we have used Minkowski's inequality, and also the fact that $ y $ is bounded in its first variable which implies $ \mathcal{K}_{j}*_{\theta}y \in L^{1}(\R) $. So by the dominated convergence theorem, we have:
		\begin{align*}
			\lim_{\epsilon\rightarrow 0}\mathcal{K}_{j}*yh &= 0.
		\end{align*}
	\end{proof}
	\begin{remark}
		If one assumes a particular form for the convolution function $ K(\xi) $, there may be weaker conditions which ensure the vanishing of the remainder terms. For instance, the uniform integrability and tightness of $ \xi^{j}K(\xi)f_{n} $ will allow for the limit as $ \epsilon\rightarrow 0 $ to commute with the integral.
	\end{remark}
	To apply \Cref{lemma:app1}, recall the remainder terms in the O$ (\epsilon) $ expansion of the convolution:
	\begin{align*}
		\mathcal{K}_{1}*_{\theta}R_{2\theta},\quad\mathcal{K}_{2}*_{\theta}R_{\psi}u_{0\theta\theta},\quad \mathcal{K}_{2}*_{\theta}R_{1\theta},\quad \mathcal{K}_{2}*_{\theta}R_{\psi}R_{1\theta}.
	\end{align*}
	Provided that the functions
	\begin{align*}
		\max_{\theta}R_{2\theta},\quad R_{\psi},\quad \max_{\theta}R_{1\theta},\quad R_{\psi}\max_{\theta}R_{1\theta}
	\end{align*}
	each satisfy the assumptions of \Cref{lemma:app1} on $ h $ for the appropriate convolution $ \mathcal{K}_{j} $, then all the remainder terms vanish from the O$ (\epsilon) $ expansion. To use the Whitham equation as an example, where
	\begin{align*}
		\widehat{K}(q) = \sqrt{\frac{\tanh q}{q}},
	\end{align*}
	we have that $ K(x) \sim \left(\frac{1}{2}\pi^{2}x\right)^{-\frac{1}{2}}e^{-\frac{1}{2}x} $ as $ x\rightarrow \infty $ \cite{Whitham1967,EW2019}. The theory we have developed would allow $ u(\theta,X,T) \in o(\sqrt{X}e^{\frac{1}{2}X}) $ as a solution to the Whitham equations. In particular, if $ u(\theta,X,T) \in \rm O(e^{(\frac{1}{2} - \beta)X}) $ as $ X \rightarrow \infty $ for $ 0 < \beta < \frac{1}{2} $, provided that $ \psi $ is chosen so that $ \xi K(\xi)R_{\psi}(X-\epsilon\xi,T) \in L^{1}(\R)$, then we can take $ R_{1},R_{2} \in {\rm O}(e^{(\frac{1}{2} - \beta)X}) $, from which the remainders converge to zero as $ \epsilon \rightarrow 0 $.
\end{appendices}

\bibliographystyle{abbrv}
\bibliography{Whitham.bib}

\end{document}